\documentclass[9pt, a4paper]{amsart}

\usepackage[usenames,dvipsnames]{xcolor}
\usepackage{amssymb,amsfonts,
			 url,amsthm, 
			 etex,cancel,pigpen,
			 enumitem,nameref,
			 stmaryrd, marginnote,
			 pifont,inconsolata,attrib,
			 lettrine, pifont,
			 epigraph, tikz, manfnt,
			 etoolbox,mparhack, microtype}
\usepackage[spanish, greek, english]{babel}
\usepackage{amsmath}
\usepackage[utf8]{inputenc}
\usepackage[T1]{fontenc}
\usepackage{hyperref}
\usepackage{lmodern}

\usepackage[titletoc]{appendix}
\usepackage[bottom=6cm,
			left=4cm,
			right=4.5cm,
			top=4.5cm]{geometry}
\usepackage[all,2cell,cmtip]{xy}\UseAllTwocells
\hypersetup{backref,
  	pdftitle={t-structures are normal torsion theories},
   pdfauthor={D. Fiorenza and F. Loregian},
	colorlinks=true,
	linkcolor=black,
	citecolor=Magenta}
\usepackage{datetime}

\newcommand{\bDelta}{\pmb\Delta}

\newcommand{\simplex}[1]{{\Delta[#1]}}

\newcommand{\mn}[1]{\ar@{>->}[#1]}
\newcommand{\ep}[1]{\ar@{->>}[#1]}

\newcommand{\xto}[1]{\xrightarrow{#1}}

\newcommand{\bsmat}[1][(]{\left#1\begin{smallmatrix}}
\newcommand{\esmat}[1][)]{\end{smallmatrix}\right#1}

\renewcommand{\phi}{\varphi}
\newcommand{\mar}{\varsigma}

\newcommand{\tsc}[1]{\textsc{#1}}
\newcommand{\tee}{\mathfrak{t}}
\def\ter{\mathcal T\!er}
\newcommand{\iso}{\textsc{Eqv}}
\newcommand{\lemname}{Sator }

\newcommand{\po}{\ar@{}[dr]|(.7){\text{\pigpenfont R}}}
\newcommand{\pb}{\ar@{}[dr]|(.3){\text{\pigpenfont J}}}
\newcommand{\pp}{\ar@{}[dr]|{\text{\pigpenfont N}}}

\newcommand{\heart}{\heartsuit}

\newcommand{\var}[2]{
	\left[\begin{smallmatrix} #1 \\ 
	\downarrow \\ #2 
	\end{smallmatrix}\right]}

	\let\amalg=\undefined
	\let\coprod=\undefined
	\DeclareSymbolFont{cmsymbols}{OMS}{cmsy}{m}{n}
	\DeclareSymbolFont{cmlargesymbols}{OMX}{cmex}{m}{n}
	\DeclareMathSymbol{\amalg}{\mathbin}{cmsymbols}{"71}
	\DeclareMathSymbol{\coprod}{\mathop}{cmlargesymbols}{"60}

\newcommand{\refbf}[1]{\textbf{\ref{#1}}}
\DeclareMathOperator{\id}{id} 			

\DeclareMathOperator{\mrk}{Mrk}

\newcommand{\cate}[1]{\text{\fontseries{b}\selectfont{#1}}}

\newcommand{\D}{\cate{D}}
\newcommand{\T}{\mathfrak{T}}
\newcommand{\F}{\mathfrak{F}}
\newcommand{\fF}{\mathbb{F}}

\newcommand{\C}{\cate{C}}
\newcommand{\E}{\mathcal{E}}
\newcommand{\M}{\mathcal{M}}

\newcommand{\fs}{\textsc{fs}}

\long\def\symbolfootnote[#1]#2{\begingroup%
\def\thefootnote{\fnsymbol{footnote}}\footnote[#1]{#2}\endgroup}

\newtheoremstyle{reference}%
   {}                %
   {}                %
   {\slshape}              
   {}                      
   {\scshape}              
   {:}                     
   {.4em}                  
   {\thmname{#1}           
    \thmnumber{#2}         
    \thmnote{{\sc [#3]}}}  

\theoremstyle{definition}
	\newtheorem{theorem}{Theorem}[section]
	\newtheorem{lemma}[theorem]{Lemma}
	\newtheorem{proposition}[theorem]{Proposition}
	
	\newtheorem{exercise}[theorem]{Exercise}
	\newtheorem{remark}[theorem]{Remark}
	\newtheorem{definition}[theorem]{Definition}
	
	\newtheorem{notat}[theorem]{Notation}
	\newtheorem*{acknowledgements}{Acknowledgements}

\begin{document}
\title{$t$-structures are normal torsion theories}
	\author{Domenico Fiorenza${}^\dag$}
	\address{${}^\dag$Dipartimento di Matematica ``Guido Castelnuovo''\\
		Universit\`a degli Studi di Roma ``la Sapienza''\\
		P.le Aldo Moro 2 -- \oldstylenums{00185} -- Roma.}
		\email{fiorenza@mat.uniroma1.it}
 
	\author{Fosco Loregi\`an${}^\ddag$}
	\address{${}^\ddag$SISSA - Scuola Internazionale Superiore di Studi Avanzati\\ 
		via Bonomea 265\\
		\oldstylenums{34136} Trieste.}
		\email{floregi@sissa.it}
		\email{tetrapharmakon@gmail.com}

	\date{\today}

\begin{abstract}
We characterize $t$-structures in stable $\infty$-categories as suitable quasicategorical factorization systems.
More precisely we show that a $t$-structure $\tee$ on a stable $\infty$-category $\C$ is equivalent to a normal torsion theory $\mathbb{F}$ on $\C$, i.e. to a factorization system $\mathbb F=(\E,\M)$ where both classes satisfy the 3-for-2 cancellation property, and a certain compatibility with pullbacks/pushouts.
\end{abstract}

\maketitle
\section{Introduction.}
The present paper aims to turn the widespread suggested connection between (torsion theories of a) \emph{reflective factorization systems} in category theory and \emph{$t$-structures} in algebraic geometry and stable homotopy theory into a precise theorem, showing how the link between these two notions exists as a genuine isomorphism (Thm. \refbf{thm:rosetta}) in the framework of stable $\infty$-categories.
 
The language used throughout the paper draws equally from (higher) category theory and homological algebra; because of its twofold nature, the ideal reader of this note is acquainted with the basic theory of both (orthogonal) factorization systems, here treated in their $\infty$-categorical counterparts presented in \cite{Joy} and \cite{HTT}, and $t$-structures in triangulated categories, for which the main references will be the classical text \cite{BBDPervers} and section \textbf{1.2} of Lurie's Higher Algebra, \cite{LurieHA}.

There is, of course, a vast literature exploring separately the classical notions of $t$-structure, torsion theory and factorization system, and yet a precise statement of (the 1-categorical counterpart of) our Thm. \refbf{thm:rosetta} seems to have eluded even comprehensive treatments like \cite{Beligiannisreiten} and \cite{CHK,RT}. Somehow mysteriously, \cite[\S \textbf{4}]{RT} seems to ignore the triangulated world, even if its authors point out clearly (see \cite[Remark \textbf{4.11.(2)}]{RT}) that
\begin{quote}
It [our definition of torsion theory, \emph{auth.}] applies, for example, to a triangulated category $\C$. Such a category has only weak kernels and weak cokernels and our definition precisely corresponds to torsion theories considered there as pairs $\F$ and $\T$ of colocalizing and localizing subcategories (see \cite{HPS}).
\end{quote}
Even more mysteriously, \cite[p. 17]{Beligiannisreiten} explicitly says that
\begin{quote}
Torsion pairs in triangulated categories are used in the literature mainly in the form of $t$-structures.
\end{quote}
and yet it avoids, in a certain sense, to offer a more primitive characterization for $t$-structures than the one given \emph{ibi}, Thm \textbf{2.13}. So, the starting point of this work can be summarized in the following question:
\begin{quote}
To which extent is it possible to prove the claim firmly suggested by the existing literature that ``$t$-structures are normal torsion theories''?
\end{quote}
Far from thinking that the authors of \cite{RT,CHK} has simply been blind to such a suggestive hypothesis, the authors believe that, since
the main result of the paper
 relies so heavily on properties only available in the stable world, there are not only conceptual, but also practical and computational reasons to adopt the setting of stable $\infty$-categories as a natural subsitute to the ``triangulated world'', and sheds a light on results otherwise unattainable or obscure.

\paragraph{\bf Organization of the paper.} Sections \textbf{2} and \textbf{3} contain introductory material about $\infty$-categorical factorization systems, stable $\infty$-categories and $t$-structures; even if a couple of results are of independent interest, they mainly serve to fix the notations we adopt in the following two sections. Section \textbf{4} is the heart of the paper, where we prove the promised characterization. The final section serves to introduce the reader to the stable $\infty$-categorical version of a few classical results in the theory of $t$-structures; a detailed discussion of Exercises 1 and 2 appears as the central result of \cite{heart}.

\paragraph{A glance to the existing literature.} There seem to be no (or better to say, too many) comprehensive references for the theory of factorization systems, since every author seems to rebuild the basic theory from scratch each time they prove a new result. Nevertheless, having to choose once and for all a reference for the interested reader, we couldn't help but mention the seminal paper by Freyd and Kelly \cite{FK}, the refined notion of ``algebraic'' factorization system proposed in Garner's \cite{Gar}, and Emily Riehl's thesis \cite{Riehl1}, whose first and second chapters, albeit being mainly interested in \emph{weak} factorization systems, constitute the best-approximation to a complete compendium about the basic theory, and finally the short, elementary note \cite{Riehl2}.

Moreover, we must mention the paper \cite{CHK} by Cassidy, H\'ebert, and Kelly, which together with \cite{RT} and the first section of \cite{Beligiannisreiten} constitute our main references for the connections between factorization systems and torsion theories in (pointed additive) categories. In particular, we point the interested reader to \cite{CHK} for a crystal-clear treatment of what we called ``fundamental connection'' in our Section \refbf{fundconn} and several adaptations of this notion in various particular contexts (pointed, well-complete and additive categories above all), and to \cite{Beligiannisreiten} for making clear that $t$-structures can be regarded as the triangulated counterpart of torsion theories in abelian categories.
\paragraph{\bf Notation and conventions.} 
Categories (and higher categories) are denoted as boldface letters $\C,\D$ etc. {Functors} between categories are always denoted as capital Latin letters like $F,G,H,K$ etc.; the category of functors $\C\to \D$ is denoted as $\text{Fun}(\C,\D)$, $\D^{\C}$, $[\C,\D]$ and suchlike; morphisms in $\text{Fun}(\C,\D)$ (i.e. natural transformations) are written in Greek alphabet. The simplex category $\bDelta$ is the \emph{topologist's delta}, having objects \emph{nonempty} finite ordinals $\simplex{n}:=\{0<1\dots<n\}$ regarded as categories in the obvious way.
We adopt \cite{HTT} as a reference for the language of quasicategories and simplicial sets; in particular, we treat ``quasicategory'' and ``$\infty$-category'' as synonyms.
\section{Quasicategorical factorization systems.}
\epigraph{[\dots] {\greektext ka`i st\'hsei t\`a m\`en pr\'obata >ek dexi\^wn a>uto\^u t\`a d\`e >er\'ifia >ex e>uwn\'umwn}.}{Matthew 25:33}
Recall that a \emph{marked simplicial set} $\underline X$ (\cite[Def. \textbf{3.1.0.1}]{HTT}) consists of a pair $(X, \mathcal S)$, where $X$ is a simplicial set, and $\mathcal S \subseteq X_1$ is a class of distinguished 1-simplices on $X$, which contains every degenerate 1-simplex.

The class of all marked simplicial sets is a category $\cate{sSet}^\mar$ in the obvious way, where a simplicial map $f\colon (X,\mathcal S_X)\to (Y, \mathcal S_Y)$ \emph{respects} the markings in the sense that $f \mathcal S_X\subseteq \mathcal S_Y$; the obvious forgetful functor
\[
U\colon \cate{sSet}^\mar \to \cate{sSet}
\]
admits both a right adjoint $X\mapsto X^\sharp =(X, X_1)$ and a left adjoint $X\mapsto X^\flat = (X, s_0(X_0))$, given by choosing the maximal and minimal markings, respectively (mnemonic trick: \textbf{r}ight adjoint is sha\textbf{r}p, \textbf{l}eft adjoint is f\textbf{l}at).
\begin{notat}
A \emph{marked quasicategory} simply consists of a marked simplicial set which, in addition, is a quasicategory. From now on, we will consider only marked quasicategories.
\end{notat}
\begin{definition}
Let $f,g$ be two edges in a quasicategory $\C$. We will say that $f$ is \emph{left orthogonal} to $g$ (or equivalently -in fact, dually- that $g$ is \emph{right orthogonal} to $f$) if in any commutative square $\simplex{1}\times\simplex{1}\to \C$ like the following,
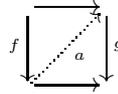
\makeatother
\[
\xymatrix{
 \ar[r]\ar[d]_f & \ar[d]^g  \\
\ar[r] \ar@{.>}[ur]_a & 
}
\]
\makeatletter
the space of liftings $a$ rendering the two triangles (homotopy) commutative is contractible\footnote{By requiring that the space of liftings $\alpha$ is only \emph{nonempty} one obtains the notion of weak orthogonality. In the following discussion we will only cope with the stronger request.}.
\end{definition}
\begin{remark}
This is Definition \cite[\textbf{5.2.8.1}]{HTT}; compare also the older \cite[Def. \textbf{3.1}]{JanelidzeMarkl}.
\end{remark}
\begin{remark}
``Being orthogonal'' defines a binary relation between edges in a marked quasicategory denoted $f\perp g$.
\end{remark}
\begin{definition}
Let $(\C, \mathcal S)$ be a marked quasicategory; we define 
\begin{gather*}
\mathcal S^\perp = \{f\colon \simplex{1}\to \C\mid s \perp f, \; \forall s\in \mathcal S\} \\
{}^\perp\mathcal S = \{f\colon \simplex{1}\to \C\mid f \perp s, \; \forall s\in \mathcal S\}.
\end{gather*}
\end{definition}
\begin{definition}[Category of markings]
If $\C$ is a quasicategory we can define an obvious posetal category $\mrk(\C)$ whose objects are different markings of $\C$ and whose arrows are given by inclusions. The maximal and the minimal markings are, respectively, the terminal and initial object of $\mrk(\C)$; this category can also be characterized as the fiber over $\C$ of the forgetful functor $U\colon \cate{sSet}^\mar \to \cate{sSet}$. 
\end{definition}
The correspondence ${}^\perp(-)\dashv (-)^\perp$ forms a Galois connection in the category of markings of $X$; the maximal and minimal markings are sent one into the other under these correspondences.
\begin{definition}
A pair of markings $(\E,\M)$ in a quasicategory $\C$ is said to be a \emph{(quasicategorical) prefactorization} when $\E = {}^\perp\M$ and  $\M = \E^\perp$. In the following we will denote a prefactorization on $\C$ as $\fF=(\E,\M)$. The collection of all prefactorizations on a given quasicategory $\C$ forms a posetal class, which we will call $\tsc{pf}(\C)$, with respect to the order $\fF=(\E,\M)\preceq \fF'= (\E',\M')$ iff $\M\subset\M'$ (or equivalently, $\E'\subset \E$).
\end{definition}
\begin{remark}
It is evident (as an easy consequence of adjunction identities) that any marking $\mathcal S\in\text{Mrk}(\C)$ induces two \emph{canonical} prefactorization on $\C$, obtained sending $\mathcal S$ to $({}^\perp\mathcal S, ({}^\perp\mathcal S)^\perp)$ and $({}^\perp(\mathcal S^\perp),\mathcal S^\perp)$. These two prefactorizations are denoted $\mathbb S_\perp$ e ${}_\perp\mathbb S$, respectively. 
\end{remark}
\begin{definition}\label{df:rlgener}
If a prefactorization $\fF$ on $\C$ is such that there exists a marking $\mathcal S\in\mrk(\C)$ such that $\fF=\mathbb S_\perp$ (resp., $\fF={}_\perp\mathbb S$) then $\fF$ is said to be \emph{right} (resp., \emph{left}) \emph{generated} by $\mathcal S$.
\end{definition}
\begin{remark}
Since the orthogonal of a class $\mathcal S$ is uniquely determined, a prefactorization is characterized by any of the two markings $\E,\M$; the class of all prefactorizations $\fF=(\E,\M)$ on a quasicategory $X=\C$ is a complete lattice whose greatest and smallest elements are respectively 
\[
(\underline X^\sharp)_\perp  = (s_0(X_0), X_1)\;\text{ and }\; {}_\perp(\underline X^\sharp)= (X_1, s_0(X_0)).\]
\end{remark}
\begin{definition}[$\fF$-crumbled morphisms]\label{def:crumble}
Given a prefactorization $\fF\in\tsc{pf}(\C)$ we say that an arrow $f\colon X\to Y$ is \emph{$\fF$-crumbled}, (or \emph{$(\E,\M)$-crumbled} for $\fF=(\E,\M)$) when there exists a (necessarily unique) factorization for $f$ as a composition $m\circ e$, with $e\in\E$, $m\in\M$; let $\sigma_\fF$ be the class of all $\fF$-crumbled morphisms, and define
\[
\tsc{pf}_{\mathcal S}(\C) = \{\fF\mid \sigma_\fF\supset\mathcal S\}\subset \tsc{pf}(\C).
\] 
\end{definition}
\begin{definition}\label{def:effe-esse}
A prefactorization system $\fF=(\E,\M)$ in $\tsc{pf}(\C)$ is said to be a \emph{factorization system} on $\C$ if $\sigma_\fF=\text{Mor}(\C)$; factorization systems, identified with $\tsc{pf}_{\text{Mor}(\C)}(\C)$, form a sublattice $\tsc{fs}(\C)\leq \tsc{pf}(\C)$.
\end{definition}
This last definition (factorizations ``crumble everything'', i.e. split every arrow in two) justifies the form of a more intuitive presentation for a \emph{(quasicategorical) factorization system} on $\C$, modeled on the classical, 1-categorical definition:
\begin{definition}[Quasicategorical Factorization System]
Let $\C$ be a quasicategory; a \emph{factorization system} (\tsc{fs} for short) $\fF$ on $\C$ consists of a pair of markings $\mathcal E, \mathcal M\in\mrk(\C)$ such that
\begin{enumerate}
\item For every morphism $h\colon X\to Z$ in $\C$ we can find a factorization $X\xto{e} Y\xto{m} Z$, where $e\in\mathcal E$ and $m\in\mathcal M$; an evocative notation for this condition is $\C = \M\circ \E$;
\item $\E ={}^{\perp}\M$ and $\M = \E^{\perp}$.
\end{enumerate}
\end{definition}
\begin{remark}
The collection of all factorization systems on a quasicategory $\C$ form a posetal category $\tsc{fs}(\C)$ with respect to the relation induced by $\tsc{pf}(\C)$.
\end{remark}
\begin{remark}
In presence of condition (1) of Definition \refbf{def:effe-esse}, the second condition may be replaced by
\begin{enumerate}
\item [(2a)] $\E \perp \M$ (namely $\E\subset {}^\perp \M$ and $\M\subset \E^\perp$);
    \item [(2b)]  $\E$ and $\M$ are closed under
    isomorphisms in $\C^{\simplex{1}}$.
\end{enumerate}
(this is precisely \cite[Def. \textbf{5.2.8.8}]{HTT}).
\end{remark}
\begin{remark}
Condition (2) of the previous Definition (or the equivalent pair of conditions (2a), (2b)) entails that each of the two classes $(\E,\M)$ in a factorization system on $\C$ uniquely determines the other (compare the analogous statement about prefactorizations): this means that the obvious functor $\tsc{fs}(\C)\to \mrk(\C)\colon (\E,\M)\mapsto \E$ is in fact a (monotone) bijection of posetal classes. This is \cite[Remark \textbf{5.2.8.12}]{HTT}.
\end{remark}
\begin{definition}[Closure operators associated to markings]\label{def:satusatu}
Let $\C$ be a quasicategory. A marking $\mathcal J\in\mrk(\C)$ is called 
\begin{itemize}
\item[\textsf{W}.)] \emph{wide} if it contains all the isomorphisms and it is closed under composition;
\end{itemize}
A wide marking $\mathcal J$ (in a quasicategory $\C$ which admits in each case the co/limits needed to state the definition) is called
\begin{itemize}
\item[\textsf{P}.)] \emph{presaturated} if is closed under co-base change, i.e. whenever we are given arrows $j\in\mathcal J$, and $h$ such that we can form the pushout
\makeatother
\[
\xymatrix@R=7mm@C=7mm{
\ar[r]^h \ar[d]_j \po &  \ar[d]^{j'} \\
 \ar[r] & 
}
\]
\makeatletter
then the arrow $j'$ is in $\mathcal J$;
\item[\textsf{Q}.)] \emph{almost saturated} if it is presaturated and closed under retracts (in the category $\C^{\simplex{1}}$), i.e. whenever we are given a diagram like
\makeatother
\[
\xymatrix@R=7mm@C=7mm{
 \ar[r]^i\ar[d]_u & \ar[r]^r\ar[d]_v  &  \ar[d]^u \\
 \ar[r]_{i'} & \ar[r]_{r'} & 
}
\]
\makeatletter
where $ri=\id_A$ and $r'i'=\id_{C}$, if $v$ lies in $\mathcal J$, then the same is true for $u$;
\item[\textsf{C}.)] \emph{cellular} if it is presaturated and closed under \emph{transfinite composition}, namely whenever we have a cocontinuous functor $F\colon \alpha\to \mathcal J$ defined from any limit ordinal $\alpha$ admits a composite in $\mathcal J$, i.e. the canonical arrow
\makeatother
\[
\xymatrix@C=1mm{
F(0) \ar[r] & **[r] F(\alpha) = \varinjlim_{i<\alpha}F(i)
}
\]
\makeatletter
lies in $\mathcal J$;
\item[\textsf{S}.)] \emph{saturated} if it is almost saturated and cellular.
\end{itemize} 
All these conditions induce suitable closure operators, encoded as suitable (idempotent) monads on $\mrk(\C)$, defined for any property $P$ among $\{ \textsf{W}, \textsf{P}, \textsf{Q}, \textsf{C}, \textsf{S} \}$ as 
\[
(-)^P \colon \mrk(\C)\to \mrk(\C)\colon \mathcal S \mapsto \mathcal S^P = \bigcap_{\mathcal{U}\supseteq \mathcal{S}}\Big\{ \mathcal U\in \mrk(\C)\mid \mathcal U \text{ has property $P$} \Big\}\]
The \emph{cellularization} $(-)^\textsf{C}$ and the \emph{saturation} $(-)^\textsf{S}$ of a marking $\mathcal J$ on $\C$ are of particular interest (especially in homotopical algebra).
\begin{notat}
A little more generality is gained supposing that the cardinality of the coproducts or the transfinite compositions in $\C$ is bounded by some (regular) cardinal $\alpha$. In this case we speak of $\alpha$-saturated or $\alpha$-cellular classes, and define the closure operators of $\alpha$-\emph{cellularization} and $\alpha$-\emph{saturation}, etc.
\end{notat}
\end{definition}
The following Proposition is a standard result in the theory of factorization systems, which we will implicitly and explicitly need all along the paper; a proof for the 1-categorical version of the statement can be found in any of the provided references about factorization systems.
\begin{proposition}\label{satu}
Let $(\C,\mathcal S)$ be a marking of the cocomplete quasicategory $\C$; then the marking ${}^\perp\mathcal S$ of $\C$ is a saturated class. In particular, the left class of a weak factorization system in a cocomplete quasicategory is saturated.
\end{proposition}
Completely dual definitions give rise to co-$P$-classes\footnote{Obviously, wideness and closure under retracts are auto-dual properties.} again, suitable mo\-nads acting as co-$P$-closure operators are defined on $\mrk(\C)$, giving the dual of Proposition \refbf{satu}.
\begin{proposition}
Let $(\C,\mathcal S)$ be a marking of the cocomplete quasicategory $\C$; then the marking $\mathcal S^\perp$ of $\C$ is a co-saturated class. In particular, the right class of a weak factorization system in a complete category is co-saturated.
\end{proposition}
\begin{proposition}\label{thereiso}
Let $\C$ be a quasicategory and $\mathbb F =(\E,\M)\in \fs(\C)$; then $\E\cap \M$ equals the class of all equivalences in $\C$.
\end{proposition}
\begin{proof}
Again, the proof in the 1-categorical case can be found in any reference about factorization systems. The idea is extremely simple: if $g\in\E\cap \M$ then it is orthogonal to itself, and the lifting problem
\makeatother
\[
\xymatrix{
\ar@{=}[r]\ar[d]_g & \ar[d]^g \\
\ar@{=}[r] &
}
\]
\makeatletter
gives a unique homotopy-inverse for $g$.
\end{proof}
\begin{definition}\label{def:3for2}
Let $\mathcal S\in\mrk(\C)$; then, for each 2-simplex in $\C$ representing a composable pair of arrows, whose edges are labeled $f,g$, and  $f g$  we say that
\begin{itemize}
\item $\mathcal S$ is \tsc{l32} if $f, fg\in \mathcal S$ imply $g\in\mathcal S$;
\item $\mathcal S$ is \tsc{r32} if $fg, g\in \mathcal S$ imply $f\in\mathcal S$.
\end{itemize}
A marking $\mathcal S$ which is closed under composition and both \tsc{l32} and \tsc{r32} is said to \emph{satisfy the 3-for-2 property}, or a \emph{3-for-2 class}.
\end{definition}
\begin{proposition}
\label{prop:clos}
Given a \tsc{fs} $(\mathcal E,\mathcal M)$ in the quasicategory $\C$, then
\begin{itemize}
\item[(i)] If the quasicategory $\C$ has $K$-colimits, for $K$ a given simplicial set, then the full subcategory of $\C^{\Delta[1]}$ spanned by $\mathcal E$ has $K$-colimits; dually, if  the quasicategory $\C$ has $K$-limits, then the full subcategory of $\C^{\Delta[1]}$ spanned by $\mathcal M$ has $K$-limits;
\item[(ii)] The class $\mathcal E$ is \tsc{r32}, and the class $\mathcal M$ is \tsc{l32} (see Def. \refbf{def:3for2}).
\end{itemize} 
\end{proposition}
\begin{proof}
Point (i) is \cite[Prop. \textbf{5.2.8.6}]{HTT}; point (ii) is easy to prove for 1-categories, and then the translation to the $\infty$-categorical setting is straightforward\footnote{This translation process being often straightforward, we choose to refer to 1-categorical sources to prove most of the result involving $\infty$-categorical factorization systems.}.
\end{proof}
It is a remarkable, and rather useful result, that each of the properties ($i$) and ($ii$) of the above Proposition characterizes factorizations among weak factorizations: see \cite[Prop. \textbf{2.3}]{RT} for more details.
\begin{remark}\label{korostenski} 
\def\due{{\simplex{1}}}
There is an equivalent presentation of the theory of factorization systems, neatly exposed in \cite{Korostenski199357} and polished by R. Garner in his \cite{Gar}, whose existence ultimately relies on the fact that the category $\due$ carries the structure of a universal comonoid $(\due, m,e)$ (see also \cite[\S \textbf{VII.5}]{McL}) as an object of $\cate{Cat}$.

We will need this characterization in section \textbf{4}, in the proof of Theorem \refbf{thm:rosetta}.
\end{remark}
\def\ter{\mathcal T\!\!er}
\subsection{\textsc{The fundamental connection.}}\label{fundconn}
Let now $\C$ be a quasicategory with terminal object $1$, and let $\ter$ be the class of the terminal morphisms $\{t_X\colon X\to 1\mid X\in\C\}$. Let also $\text{Rex}(\C)$ be the poset of reflective subcategories $(\cate{B}, R)$ of $\C$ (where $R\colon \C \to \cate B$ is the reflection functor, left adjoint to the inclusion).

We now want to reproduce the construction at the beginning of \cite{CHK}, where the authors build a correspondence between $\tsc{pf}_{\ter}(\C)$ (notations as in Definition \refbf{def:crumble}) and $\text{Rex}(\C)$.
\begin{proposition}\label{connectio}
There exists a(n antitone) Galois connection $\Phi\dashv \Psi$ between the posets $\text{Rex}(\C)$ and $\tsc{pf}_{\ter}(\C)$, where $\Psi$ sends $\fF=(\E,\M)$ to the subcategory $\M/1 = \{B\in\C\mid (B\to 1)\in\M\}$, and $\Phi$ is defined sending $(\cate B,R)\in\text{Rex}(\C)$ to the prefactorization \emph{right generated} (see Definition \refbf{df:rlgener}) by $\hom(\cate B)$. 
\end{proposition}
\begin{remark}\label{funtoriali}
The action of the functor $R\colon \C\to \M/1$ is induced on objects by a choice of $\mathbb F$-factorizations of terminal morphisms: $X\xto{e} RX\xto{m} 1$. On arrows it is obtained from a choice of solutions to lifting problems
\[
\xymatrix{
A \ar[r]^{ef}\ar[d] & RB\ar[d]^m \\
RA \ar[ur]_{Rf}\ar[r]_{m}& 1.
}
\]
\end{remark}
\begin{remark}
The unit $\id_{\text{Rex}(\C)}\Rightarrow\Psi\Phi$ of this adjunction is an isomorphism.

The comonad $\Phi\Psi\Rightarrow \id_{\tsc{pf}_{\ter}(\C)}$ is much more interesting, as it acts like an \emph{interior operator} on the poset $\tsc{pf}_{\ter}(\C)$, sending $\fF$ to a new prefactorization $\mathring{\fF}=(\mathring{\E},\mathring{\M})$ which is by construction \emph{reflective}, i.e. satisfies $\mathring{\fF}=\fF$ (whereas in general we have only a proper inclusion). 
\end{remark}
What we said so far entails that
\begin{proposition}
The adjunction $\Phi\dashv\Psi$ restricts to an equivalence (a bijection between posets) between the reflective prefactorizations in $\fF\in \tsc{pf}_{\ter}(\C)$ and the poset $\text{Rex}(\C)$.
\end{proposition}
\begin{proposition}\label{refective}
$\fF\in \tsc{pf}_{\ter}(\C)$ is reflective if and only if $\E$ is a 3-for-2 class (see Definition \refbf{def:3for2}), or equivalently (since each $\E$-class of a factorization system is \tsc{r32}) if and only if $\E$ has the half of the 3-for-2 property it lacks.
\end{proposition}
\begin{proof}
It is an immediate consequence of \cite[Thm. \textbf{2.3}]{CHK}, where it is stated that $g\in\mathring{\E}$ iff $fg\in\E$ for a suitable $f\in \E$.
\end{proof}
We can also state completely dual results about \emph{co}reflective subcategories, linked to (pre)\-fac\-to\-ri\-za\-tion systems factoring at least \emph{initial} arrows in $\C$ via the correspondence $\fF \mapsto \varnothing/\E = \{Y\in\C\mid (\varnothing \to Y)\in\E\}$; the coreflection of $\C$ along $\varnothing/\E$ is given by a functor $S$ defined by a choice of $\mathbb F$-factorization $\varnothing\xto{e}SX\xto{m}X$.
\begin{remark}\label{biref}
We can also define \emph{co}reflective factorization systems, and prove that $\fF$ is coreflective iff $\M$ is \tsc{r32}, and \emph{bi}reflective factorization systems as those which are reflective \emph{and} coreflective at the same time.
\end{remark}
\subsection{\textsc{Semiexact and simple factorization systems.}}
A fairly general theory stems from the above construction, and several peculiar classes of factorization systems become of interest, aside from (co)reflective ones:
\begin{definition}\label{slex}
 A \emph{semi-left-exact} factorization system on a finitely complete $\C$ consists of a reflective $\fF=(\E,\M)\in \tsc{fs}(\C)$ such that the left class $\E$ is closed under pulling back  by $\M$ arrows; more explicitly, in the pullback
\[
\xymatrix{
\ar[r] \ar[d]_{e'}\pb & \ar[d]^{e\in\E} \\
\ar[r]_{m\in\M} &
}
\]
the arrow $e'$ lies in $\E$.
\end{definition}
Equivalent conditions for $\fF$ to be semi-left-exact are given in \cite[Thm. \textbf{4.3}]{CHK}. There is a dual definition of a semi-\emph{right}-exact factorization system. We call \emph{semiexact} a factorization system which is both left and right exact.
\begin{definition}\label{simple}
A \emph{left simple} factorization system on $\C$ consists of $\fF\in \tsc{fs}(\C)$ whose factorization action on arrows of $\C$ goes as follows: if we denote by $R$ the reflection (having unit $\eta$) $\C\to \M/1$ associated to $\fF$ via the functor $\Psi$, then the $\fF$-factorization of $f\colon X\to Y$ can be obtained as $X\to RX\times_{RY}Y\to Y$ in the diagram
\makeatother
\[
\xymatrix{
X \ar[dr]\ar@/^1pc/[drr]^{\eta_X}\ar@/_1pc/[ddr]_f& \\
& RX\times_{RY}Y \pb \ar[r]\ar[d]& RX\ar[d]^{Rf}\\
& Y \ar[r]_{\eta_Y} & RY
}
\]
\makeatletter
obtained from the naturality square for $f$. \end{definition}
\begin{remark}
Every semi-left-exact factorization system is left simple, as proved in \cite[Thm. \textbf{4.3}]{CHK}. In the 1-categorical setting, the converse doesn't hold in general (see \cite[Example \textbf{4.4}]{CHK}), whereas our Prop. \refbf{simplenormalexact} shows that in the stable $\infty$-categorical world the two notions coincide.
\end{remark}
\begin{remark}
There is an analogous notion of \emph{right simple} factorization system: semi-right-exact factorization systems are right simple.
\end{remark}
An useful result follows from the semi-exactness of a factorization system $\fF$ whose both classes are 3-for-2: these are called \emph{torsion theories} in \cite{RT} (see our Def. \refbf{tortorfree} for an extensive discussion).
\begin{proposition}
\label{facto}Let $\fF$ be a torsion theory whose reflection is $R$  and whose coreflection is $S$: then we have that
\[
SY \amalg_{SX}X \cong RX \times_{RY}Y 
\]
for any $f\colon X\to Y$.
\end{proposition}
\begin{proof}
The claim holds simply because  left semiexactness gives the $\fF$-factorization of $f\colon X\to Y$ as $X\to RX \times_{RY}Y \to Y$, and right semiexactness gives $X\to SY \amalg_{SX}X\to Y$.

But there is a more explicit argument which makes explicit use of the orthogonality and 3-for-2 property: consider the diagram
\makeatother
\[
\xymatrix@R=2mm{
SX \ar@{}[ddr]|(.7){\text{\pigpenfont R}}\ar[r]^{\sigma_X}\ar[dd]_{Sf}& X \ar@/^1pc/[drrr]^{\eta_X}\ar[drr]\ar[dddrr]^f\ar[dd]& \\
&&&P \ar@{}[ddr]|(.3){\text{\pigpenfont J}}\ar[r]\ar[dd] & RX\ar[dd]^{Rf} \\
SY \ar@/_1pc/[drrr]_{\sigma_Y}\ar[r]& Q\ar[drr] &&\\
&&& Y \ar[r]_{\eta_Y}& RY
}
\]
\makeatletter
where $\eta$ is the unit of the reflection $R$, and $\sigma$ is the counit of the coreflection $S$. Now the arrow $\var{X}{Q}$ is in $\E$, and the arrow $\var{P}{Y}$ is in $\M$, as a consequence of stability under cobase and base change (see Prop. \refbf{satu}); this entails that there is a unique $w\colon Q\to P$ making the central square commute. Now, semiexactness entails that $X\to P\to Y$ and $X\to Q\to Y$ are both $\mathbb F$-factorizations of $f\colon X\to Y$, and since both classes $\E,\M$ are 3-for-2, we can now conclude that $w\colon Q\to P$ lies in $\E\cap\M$, and hence is an equivalence (see Prop. \refbf{thereiso}).
\end{proof}
\section{Stable {$\infty$}-categories.}
 \epigraph{\lettrine{O}{tra escuela} declara [\dots] que nuestra vida es apenas el recuerdo o reflejo crepuscular, y sin duda falseado y mutilado, de un proceso irrecuperable.}{J. L. Borges, \emph{Tl\"on, Uqbar, Orbis Tertius}}
Our aim in this section is to specialize the above definitions to the case of a \emph{stable $\infty$-category} in the sense of Lurie's \cite{LurieHA}, in order to present the main result of this note:
\begin{theorem}\label{thm:rosetta0}
$t$-structures in the sense of \cite[Def. \textbf{1.2.1.4}]{LurieHA} correspond to \emph{normal} factorization systems in the stable $\infty$-category $\C$, as in Definition \refbf{normal}.
\end{theorem} 
Given our particular interest, we will now recall those features of theory of stable $\infty$-categories which will be relevant to this note. An extensive treatment can be found in \cite{LurieHA}.

\subsection{\textsc{Triangulated higher categories.}} Pathological examples aside (see \cite{FMS}, from which the following distinction is taken verbatim), there are essentially two procedures to build ``nice'' triangulated categories:
\begin{itemize}
\item In Algebra they often arise as the stable category of a Frobenius category (\cite[\textbf{4.4}]{shc}, \cite[\textbf{IV.3} Exercise \textbf{8}]{Gel}).
\item In algebraic topology they usually appear as a full triangulated subcategory of the homotopy category of a Quillen stable model category \cite[\textbf{7.1}]{Hov}.
\end{itemize}
The [closure under equivalence of] these two classes contain respectively the so-called \emph{algebraic} and \emph{topological} triangulated categories described in \cite{Sch}.
Classical triangulated categories can also be seen as Spanier-Whitehead stabilizations of the homotopy category $\text{Ho}(\cate{M})$ of a pointed model category $\cate M$ (see \cite{DeA} thesis for an exhaustive treatment of this construction).

Because of this remark, analyzed also in \cite[Ch. \textbf{7}]{Hov}, \emph{stable model categories} can be thought of as counterparts to triangulated categories in the higher-categorical world.

Several different models for higher-dimensional analogues of triangulated categories arose as a reaction to different needs in abstract Homological Algebra (where the paradigmatic example of such an object is the derived categories of chain complexes of modules on a ring), algebraic geometry (where one is led to study derived categories of coherent sheaves on spaces) or in a fairly non-additive setting as algebraic topology (where the main example of such a structure is the homotopy category of spectra); there's no doubt that allowing a certain play among different models may be more successful in describing a particular phenomenon (or a wider range of phenomena), whereas being forced to a particular one may turn out to be insufficient.

Now, a ``principle of equivalence'' in higher category theory tells us that there must be an equivalent formulation (or better, \emph{presentation}) of triangulated $\infty$-categories in terms of quasicategory theory, such that when a quasicategory $\C$ enjoys a property which \cite{LurieHA} calls ``stability'', then
\begin{itemize}
\item its homotopy category $\text{Ho}(\C)$ is triangulated structure in the classical sense;
\item the axioms characterizing a triangulated structure are ``easily verified and well-motivated consequences of evident universal arguments'' (see \cite[Remark \textbf{1.1.2.16}]{LurieHA}) living in $\C$;
\item classical derived categories arising in Homological Algebra can be regarded as homotopy categories of stable $\infty$-categories functorially associated to an abelian $\mathcal A$ (see \cite[\S \textbf{1.3.1}]{LurieHA}).
\end{itemize}
Building this theory is precisely the aim of \cite[Ch. \textbf{1.1}]{LurieHA}. We now want to give a rapid account of its main lines.

We invite the reader to take \cite{LurieHA} as a permanent reference for this section, hoping to convince those already acquainted with the theory of triangulated categories that they are already able to manipulate the entire theory of stable $\infty$-categories even if they don't know.
\subsection{\textsc{Stable quasicategories.}}
Let $\simplex{1}\times \simplex{1}$ be the category \[
\xymatrix{
(0,0) \ar[r]\ar[d]& (0,1)\ar[d]\\
(1,0) \ar[r] & (1,1)
}
\] and denote it as $\square$ for short. It is obvious that $\text{Map}(\square, \C)$ consists of commutative squares in $\C$. This said we can give the following
\begin{definition}[(Co)cartesian square]
A diagram $F\colon \square\to \C$ in a (finitely bicomplete) quasicategory is said to be \emph{cocartesian} (resp., \emph{cartesian}) if the square
\[
\xymatrix{
F(0,0) \ar[r]\ar[d]& F(0,1) \ar[d]\\
F(1,0) \ar[r] & F(1,1)
}
\]
is a homotopy pushout (resp., a homotopy pullback)
\end{definition}
\begin{definition}[Stable quasicategory]\label{def:stablequasi}
A quasicategory $\C$ is called \emph{stable} if
\begin{enumerate}
\item it has any finite (homotopy) limit and colimit;
\item A square $F\colon\square \to \C$ is cartesian if and only if it is cocartesian.
\end{enumerate}
\end{definition}
\begin{notat}
Squares which are both pullback and pushout are called \emph{pulation squares} or \emph{bicartesian squares} (see \cite[Def. \textbf{11.32}]{acc}) in the literature.
We choose to call them \emph{pullout squares} and we refer to axiom \textbf{2} above as the \emph{pullout axiom}: in such terms, a stable quasicategory is a finitely bicomplete quasicategory satisfying the pullout axiom.
\end{notat}
Most of the arguments in the following discussion are a consequence of a fundamental remark:
\begin{remark}
The pullout axiom implies that the class $\mathcal P$ of pullout squares in a category $\C$ satisfies a 3-for-2 property: in fact, it is a classical result (see \cite[Prop. \textbf{11.10}]{acc} and its dual) that pullback squares have \tsc{r32} property and dually, pushout squares have \tsc{l32} property (these are called \emph{pasting laws} for pullback and pushout squares) in the sense of our Definition \refbf{def:3for2} when regarded as morphisms in the category $\C^{\simplex{1}}$.
\end{remark}
\begin{notat}
It is a common practice to denote diagrammatically a (co)cartesian square ``enhancing'' the corner where the universal object sits (this well-established convention has been used freely in the previous sections): as a ``graphical'' representation of the auto-duality of the pullout axiom, we choose to denote a pullout square enhancing \emph{both} corners: 
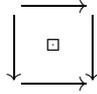
\makeatother
\[
\xymatrix{
\pp \ar[r]\ar[d] & \ar[d]\\
\ar[r] & 
}
\]
\makeatletter
\end{notat}
\begin{remark}[The pullout axiom induces an enrichment.]
What we called the pullout axiom in Definition \refbf{def:stablequasi} is an extremely strong assumption\footnote{So strong that it becomes trivial in low dimensions: it's easy to see that a 1-category $\C$ where a square is a pullback if and only if it is a pushout is forced to be the terminal category $\cate{1}$.} which, taken alone, characterizes almost completely the structure of a stable $\infty$-category. 

For instance, by invoking basically only the pullout axiom, one can prove that a stable quasicategory $\C$
\begin{itemize}
\item has a zero object, i.e. there exists an arrow $1\to\varnothing$ (which is forced to be an isomorphism);
\item $\C$ has biproducts, i.e. $X\times Y\simeq X\amalg Y$ for any two $X,Y\in\C$, naturally in both $X$ and $Y$. 
\end{itemize}
(the interested reader can take this statement as a little challenge to test the power of the pullout axioms, and their understanding of the theory so far).
\end{remark}
\subsection{\textsc{Loops and suspensions.}} The \emph{suspension} $\Sigma X$ of an object $X$ in a finitely cocomplete, pointed quasicategory $\C$ is defined as the (homotopy) colimit of the diagram $0\leftarrow X\to 0$; dually, the \emph{looping} $\Omega X$ of an object $X$ in such a $\C$ is defined as the (homotopy) limit of $0\to X\leftarrow 0$. 

This notation is natural if one thinks to the category of pointed spaces, where this operation amounts to the well-known \emph{reduced suspension} of $X$, $\Sigma\colon X\mapsto X\land S^1$; evaluating a square $F\colon \square\to \C$ at its right-bottom vertex gives an endofunctor $\Sigma\colon \C\to \C$, and where the looping $\Omega$ is the right adjoint of $\Sigma$. We depict the objects $\Sigma X, \Omega X$ as vertices of the diagrams
\makeatother
\[
\xymatrix{%
X\ar[r]\ar[d]\po & 0\ar[d] & \Omega X \pb \ar[r]\ar[d] & 0\ar[d]\\
0\ar[r] & \Sigma X & 0\ar[r] & X
}
\]
\makeatletter
The pullout axiom defining a stable quasicategory implies that these two correspondences (which in general are adjoint functors between quasicategories: see \cite[Remark \textbf{1.1.2.8}]{LurieHA}) are a pair of mutually inverse equivalences (\cite[Prop. \textbf{5.8}]{Gro}).
\begin{notat}
In a stable setting, we will often denote the image of $X$ under the suspension $\Sigma$ as $X[1]$, and by extension $X[n]$ will denote, for any $n\ge 2$ the object $\Sigma^nX$ (obviously, $X[0]:=X$). Dually, $X[-n]:=\Omega^n X$ for any $n\ge 1$.

This notation is in line with the long tradition to denote $X[1]$ the \emph{shift} of an object $X$ in a triangulated category; \emph{distinguished triangles}, often denoted as
\[
X \to Y \to Z \to X[1]
\]
(or $X\to Y\to Z\to^+$ for short) in the triangulated world, are called (again distinguished triangles or) \emph{fiber sequences} in the stable world (see \cite[Def. \textbf{1.1.2.11}]{LurieHA}) and depicted as pullout squares
\[
\xymatrix{
X \pp \ar[r]\ar[d]& Y \pp \ar[r]\ar[d]& 0\ar[d] \\
0\ar[r] & Z\ar[r] & W
}
\]
(again, this terminology is clear having in mind the topological example of the category of spectra) The pullout axiom now entails that $W\cong X[1]$.
\end{notat}
The definitions given so far amount to a process of \emph{enhancement} of the classical theory of triangulated category: one of the most unsatisfactory features of the classical theory (at least, for a category theorist\dots) is that the well-known localization procedures used to build them \emph{destroy} even simple limits and colimits. One of the advantages of the theory exposed so far is that instead, now we are working at a prior stage, where these limits still exist (Definition \refbf{def:stablequasi}, axiom \textbf{1}) and are extremely peculiar (Definition \refbf{def:stablequasi}, axiom \textbf{2})\footnote{Albeit seldom spelled out explicitly, we can trace in this remark a fundamental tenet of the theory exposed in \cite{LurieHA}:
\begin{quote}
In the same way every shadow comes from an object, produced once the sun sheds a light on it, every ``non-pathological'' triangulated category is the 1-dimensional shadow (i.e. the homotopy category) of an higher-dimensional object.
\end{quote}
No effort is made here to hide that this fruitful metaphor is borrowed from \cite{Olivia}, even if with a different meaning and in a different context.}.
\subsection{\textsc{$t$-structures.}} We can now address the main aim of our work, the investigation of $t$-structures in stable $\infty$-categories. Our reference for the classical theory in triangulated categories are the book \cite{KS1} and the classical \cite{BBDPervers}; the $\infty$-categorical analogue of the theory has been defined by Lurie in \cite[\S \textbf{1.2.1}]{LurieHA}. We now merely recall a couple of definitions for the ease of the reader: from \cite[Def. \textbf{1.2.1.1} and \textbf{1.2.1.4}]{LurieHA} one obtains
\begin{definition}\label{tistru}
Let $\C$ be a stable quasicategory. A \emph{$t$-structure} on $\C$ consists of a pair $\tee=(\C_{\ge 0},\C_{< 0})$ of full sub-quasicategories satisfying the following properties:
\begin{itemize}
\item[(i)] orthogonality: $\C(X, Y)$ is a contractible simplicial set for each $X\in \C_{\ge 0}$, $Y\in \C_{< 0}$;
\item[(ii)] Setting $\C_{\geq 1}=\C_{\geq 0}[1]$ and $\C_{<-1}= \C_{<0}[-1]$ one has $\C_{\geq 1}\subseteq \C_{\geq 0}$ and $\C_{<-1}\subseteq \C_{<0}$;
\item[(iii)] Any object $X\in\C$ fits into a (homotopy) fiber sequence $X_{\ge 0}\to X\to X_{< 0}$, with $X_{\ge 0}$ in $\C_{\ge 0}$ and $X_{<0}$ in $\C_{< 0}$. 
\end{itemize}
\end{definition}
\begin{remark}
The assignments $X\mapsto X_{\ge 0}$ and $X\mapsto X_{< 0}$ define two functors $\tau_{\ge 0}$ and $\tau_{<0}$ which are, respectively, a right adjoint to the inclusion functor $\C_{\ge 0}\hookrightarrow \C$ and a left adjoint to the inclusion functor $\C_{<0}\hookrightarrow \C$. In other words, $\C_{\ge 0},\C_{<0} \subseteq \C$ are respectively a coreflective and a reflective subcategory of $\C$: see \cite[\textbf{1.2.1.5-8}]{LurieHA} this in particular implies that 
\begin{itemize}
\item The full subcategories $\C_{\ge n}=\C_\ge[n]$, are coreflective via a coreflection $\tau_{\ge n}$; dually $\C_{<n}=\C_{<0}[n]$ are reflective via a reflection  $\tau_{< n}$, 
\item $\C_{< n}$ is stable under all limits which exist in $\C$, and colimits are computed by applying the reflector $\tau_{< n}$ to the colimit computed in $\C$; dually, $\C_{\ge n}$ is stable under all colimits, and limits are $\C$-limits coreflected via $\tau_{\ge n}$; in particular $\tau_{<n}$ maps a pullout in $\C$ to a pushout in $\C_{<n}$ while $\tau_{\geq n}$ maps a pullout in $\C$ to a pullback in $\C_{\geq n}$.
\end{itemize}
\end{remark}
\begin{notat}
An important notational remark: the subcategory that we here denote $\C_{<0}$ is the subcategory which would be denoted $\C_{\leq 0}[-1]$ in \cite{LurieHA}.
\end{notat}
\begin{remark}
It's easy to see that Definition \refbf{tistru} is modeled on the classical definition of a $t$-structure (\cite{KS1}, \cite{BBDPervers}). In fact a $t$-structure $\tee$ on $\C$, following \cite{LurieHA}, can also be characterized as a $t$-structure (in the classical sense) on the homotopy category of $\C$ (\cite[Def. \textbf{1.2.1.4}]{LurieHA}), once $\C_{\ge 0}, \C_{<0}$ are identified with the subcategories of the homotopy category of $\C$ spanned by those objects which belong to the (classical) $t$-structure $\tee$ on the homotopy category. \end{remark}
\begin{remark}
\marginnote{\textdbend} 
The notation $\C_{\geq 1}$ for $\C_{\geq 0}[1]$ is powerful but potentially misleading: namely one is lead to figure $\C_{\geq 0}$ a the seminfinite interval $[0,+\infty)$ in the real line and $\C_{\geq 1}$ as the seminfinite interval $[1,+\infty)$. This is indeed a very useful analogy (see Remark \ref{evocative}) but one should always keep in mind that as a particular case of the inclusion condition $\C_{\geq 1}\subseteq \C_{\geq 0}$ also the extreme case $\C_{\geq 1}=\C_{\geq 0}$ is possible, in blatant contradiction of the real line semintervals mental picture.
\end{remark}
\begin{remark}\label{slicing}
The collection $\tsc{ts}(\C)$ of all $t$-structures on $\C$ has a natural posetal structure by $\tee\preccurlyeq \tee'$ if $\C_{<0}\subseteq \C'_{<0}$. The ordered group $\mathbb{Z}$ acts on $\tsc{ts}(\C)$ with the generator $+1$ mapping a $t$-structure $\tee=(\C_{\geq0},\C_{<0})$ to the $t$-structure $\tee[1]=(\C_{\geq 1},\C_{<1})$. Since $\tee\preccurlyeq\tee[1]$ one sees that $\tsc{ts}(\C)$ is naturally a $\mathbb{Z}$-poset. It is therefore meaningful to consider \emph{families} of $t$-structures on $\C$ indexed by a $\mathbb{Z}$-poset $J$, i.e., $\mathbb{Z}$-equivariant morphisms of posets $J\to \tsc{ts}(\C)$. In particular, for $J=\mathbb{R}$ one recovers Bridgeland's notion of \emph{slicing} \cite{Brid}, see \cite{GKR}. A more detailed discussion of slicings in stable $\infty$-categories will hopefully appear elsewhere. 
\end{remark}
\begin{remark}
Alternatively (\cite[Prop. \textbf{1.2.1.16}]{LurieHA}) a $t$-structure $\tee$ on $\C$ is completely determined by a $t$-\emph{localization} $L$, i.e. by a reflection functor $L$ satisfying one of the following equivalent properties:
\begin{itemize}
\item The class of $L$-local morphisms\footnote{An arrow $f$ in $\C$ is called \emph{$L$-local} if it is inverted by $L$; it's easy to see that $L$-local objects form a quasisaturated class in the sense of \cite[Def. \textbf{1.2.1.14}]{LurieHA}.} is generated (as a quasisaturated marking) by a family of initial arrows $\{0\to X\}$;
\item The class of $L$-local morphisms is generated (as a quasisaturated marking) by the class of initial arrows $\{0\to X\mid LX\simeq 0\}$; 
\item The essential image $L\C\subset\C$ is an extension-closed class.
\end{itemize}
The $t$-structure $\tee(L)$ determined by the $t$-localization $L\colon \C\to \C$ is given by the pair of subcategories
\[
\C_{\ge 0}(L) := \{A\mid LA\simeq 0\},\qquad \C_{<0}(L) := \{B\mid LB\simeq B\} .
\]
It is no surprise that the obvious example of $t$-localization is the truncation $\tau_{<0}:\C\to\C_{<0}$ associated with a $t$-structure $(\C_{\geq0},\C_{<0})$, and that one has $\C_{\ge 0}(\tau_{<0})=\C_{\ge 0}$ and $\C_{< 0}(\tau_{<0})=\C_{< 0}$.
\end{remark}
This connection is precisely what motivated us to exploit the theory of factorization systems to give an alternative description of the data contained in a $t$-structure: the synergy between orthogonality encoded in property (i) of Definition \refbf{tistru} and reflectivity of the subcategories generated by $\tee$, suggest to translate in the language of (stable) $\infty$-categories the content of \cite{RT} and \cite{CHK}, on whose backbone we build the rest of the paper
\section{{$t$}-structures are factorization systems.}
\epigraph{\lettrine{A}{caso} un arquetipo no revelado a\'un a los hombres, un objeto eterno (para usar la nomenclatura de Whitehead), est\'e ingresando paulatinamente en el mundo; su primera manifestaci\'on fue el palacio; la segunda el poema. Quien los hubiera comparado habr\'ia visto que eran esencialmente iguales.}{J. L. Borges, \emph{El sue\~no de Coleridge}}
This is the gist of the paper, where we provide a detailed proof of the result previewed on page \pageref{thm:rosetta0}: the following section is entirely devoted to a complete, exhaustive proof that normal factorization systems correspond to $t$-structures on a stable quasicategory. We begin introducing the former notion.
\subsection{\textsc{Normal torsion theories.}}
Following (and slightly adapting) \cite[\S 4]{RT} we give the following definition. From now on $\C$ will denote a stable $\infty$-category, with zero object $0$.
\begin{definition}[Torsion theory, torsion classes]\label{tortorfree}
A \emph{torsion theory} in $\C$ consists of a factorization system $\fF=(\E,\M)$ (see Remark \refbf{biref}), where both classes are 3-for-2 (in the sense of Definition \refbf{def:3for2}). We define
 $\T = 0/\E$ and $\F =  \M/0$ (see Prop. \refbf{connectio}) to be respectively the \emph{torsion} and \emph{torsion-free} classes associated to the torsion theory.
\end{definition}
\begin{remark}
In view of Prop. \refbf{refective} and its dual, the torsion and torsion-free classes of a torsion theory $\fF\in \textsc{fs}(\C)$ are respectively a coreflective and reflective subcategory of $\C$.

If we $\fF$-factor the terminal and initial morphisms of any object $X\in \C$, we obtain the the reflection $R\colon \C\to \M/0$ and coreflection $S\colon \C \to 0/\E$, and a ``complex'' 
\[\tag{$\star$}
\xymatrix{
 SX\ar[r] & X\ar[r] & RX
}
\]
(in the sense of pointed categories), i.e., a homotopy commutative diagram 
\[
\xymatrix{
 SX\ar[r] \ar[d]& X\ar[d]\\
 0\ar[r] & RX
}
\]
as it is immediately seen by the orthogonality condition.
\end{remark}
\begin{lemma}\label{orthoreflex}
Let $(\C,\fF)$ be a $\infty$-category endowed with a torsion theory. Then the following conditions are equivalent:
\begin{enumerate}
\item $A\in \T = 0/\E$;
\item $\C(A, X)$ is contractible for each $X\in \F=\M/0$;
\item $RA=0$.
\end{enumerate}
In particular, one has $RSX=0$ for every $X\in \C$.
\end{lemma}
\begin{proof}
We adapt to the stable $\infty$-categorical setting the proof found in \cite{RT}, which states a directly analogous result.

(i) $\Rightarrow$ (ii). If $A\in\T$, the space of solutions of the lifting problem \makeatother \[\xymatrix@R=5mm@C=5mm{0 \ar[r]\ar[d]& B \ar[d]\\
A \ar[r]& 0}\] \makeatletter must be contractible for any $B\in\F$, and yet it coincides with the whole $\C(A,B)$.

(ii) $\Rightarrow$ (iii). Factoring $A\to 0$ as $A\xto{\rho_A}RA\to 0$ we get that $\rho_A=0_A$; but now the diagram
\makeatother
\[
\xymatrix@R=4mm{
A \ar[r]\ar[d]_{u_A}& RA \ar[dd]^{u_{RA}}\\
0 \ar[d]& \\
RA \ar[r]_{u_{RA}}\ar@{=}[uur]& 0
}
\]
\makeatletter
commutes if we call $u_A\colon A\to 0$, $v_A\colon 0\to A$ the terminal and initial natural transformations respectively. Hence, 
\[
v_{RA}u_Av_Au_{RA} = v_{RA}u_{RA} = 1_{RA}
\]
from which we deduce that the identity of $RA$ is homotopic to the zero ma, so that $RA\cong 0$. The fact that (iii) implies (i) is evident, and this concludes the proof.
\end{proof}

There is, obviously, a dual result, stated as
\begin{lemma}
In the same notations as Definition \refbf{tortorfree}, $\M=S^{-1}(\iso)$ if and only if $\M$ is a 3-for-2 class in the sense of Definition \refbf{def:3for2}. In this case,
\[
\F = \{ A\in \C\mid SA\cong 0\}
\]
coincides with the collection of all $Y\in\C$ such that $\C(A, Y)$ is contractible for each $A\in \T$.
\end{lemma}

\begin{remark}
Given the closure properties of the classes $\E, \M$, we can define natural functors $F\colon \C\to \F$ and $T\colon \C\to \T$ by taking suitable pullbacks and pushouts. Namely, we define $FX$ as the homotopy pullback
\makeatother
\[\xymatrix@R=5mm@C=5mm{FX \pb \ar[r]\ar[d]& SX\ar[d] \\ 0\ar[r] & X}\]
\makeatletter
and $TX$ as the homotopy pushout
\makeatother
\[\xymatrix@R=5mm@C=5mm{X  \ar[r]\ar[d]\po& 0\ar[d] \\ RX\ar[r] & TX}\]
\makeatletter
This construction is so natural it may show up in unexpected situations: for instance, \cite[Def. \textbf{3.9.12}]{schreiber} shows how the coefficient object $\flat_{\mathrm{dR}} A$ for de Rham cohomology with coefficients in an abelian group $A$ can be realized as the homotopy fiber of the coreflection $\flat A\to A$; this allows for defining de Rham cohomology in an arbitrary cohesive $\infty$-topos. 
\end{remark}
In what follows we will be mainly concerned with the ``other way'' to take fibers and cofibers, i.e., in the \emph{fiber} $KX$ of the \emph{reflection} $X\to RX$ and in the \emph{cofiber} $QX$ of the \emph{coreflection} $SX\to X$. When this procedure is well-behaved, we obtain the key notion of \emph{normality} of a torsion theory, which we introduce below.

As previewed at the end of our section \textbf{2}, a fairly general theory stems from the fundamental connection established in \cite{CHK}, and several specializations of a factorization system on $\C$ capture different kinds of reflective subcategories of $\C$ under this construction. We are particularly interested in the properties of the class of those factorization systems called \emph{normal} in \cite{CHK} and \cite{RT}. These can be defined intuitively as the torsion theories $\fF = (\E,\M)$ such that the diagram induced in ($\star$) is an ``exact sequence'', i.e., such that the diagram
\[
\xymatrix{
 SX \pp \ar[r] \ar[d]& X\ar[d]\\
 0\ar[r] & RX
}
\]
is a pullout. This seems to shed a light on \cite[Remark \textbf{7.8}]{CHK} and \cite[Remark \textbf{4.11}]{RT}, where the non-existence of a non-artificial example of a non-normal torsion theory is conjectured.  However this characterization is not immediate, and it admits a certain number of equivalent reformulations (see Prop \refbf{equivcondnorm}).
\begin{definition}\label{normal} We call \emph{left normal} a torsion theory $\fF=(\E,\M)$ on $\C$ such that the fiber $KX\to 0$ of a reflection morphism $X\to RX$ lies in $\E$, as in the diagram 
\makeatother
\[\xymatrix@R=5mm@C=5mm{KX \pb \ar[r]\ar[d]& X\ar[d] \\ 0\ar[r] & RX}\]
\makeatletter
In other words, the $\E$-morphisms arising as components of the unit $\eta\colon 1\Rightarrow R$ are stable under pullback along the initial $\M$-morphism $0\to RX$.
\end{definition}
\begin{remark}
This last sentence deserves a deeper analysis: by the very definition of $RX$ it is clear that $RX\to 0$ lies in $\M$; but more is true (and this seemingly innocuous result is a key step of most of the proofs we are going to present): since $\M$ enjoys the 3-for-2 property, and it contains all isomorphisms of $\C$, it follows immediately that an initial arrow $0\to A$ lies in $\M$ \emph{if and only if} the terminal arrow $A\to 0$ on the same object lies in $\M$. The same reasoning applied to $\E$ gives a rather peculiar ``specularity'' property for both classes $\E,\M$:
\begin{lemma}[\lemname Lemma]\label{satorlemma}
\makeatother
In a pointed quasicategory $\C$, an initial arrow $0\to A$ lies in a class $\E$ or $\M$ of a bireflective (see Remark \refbf{biref}) factorization system $\fF$ if and only if the terminal arrow $A\to 0$ lies in the same class.\footnote{The so-called \emph{Sator square}, first found in the ruins of Pompeii, consists of the $5\times 5$ matrix 
\[
\xymatrix@R=-1mm@C=-1mm{ 
\tsc{s} & \tsc{a} & \tsc{t} & \tsc{o} & \tsc{r}\\
 \tsc{a} & \tsc{r} & \tsc{e} & \tsc{p} & \tsc{o}\\
\tsc{t} & \tsc{e} & \tsc{n} & \tsc{e} & \tsc{t}\\
\tsc{o} & \tsc{p} & \tsc{e} & \tsc{r} & \tsc{a}\\
\tsc{r} & \tsc{o} & \tsc{t} & \tsc{a} & \tsc{s}
}
\]
where the letters are arranged in such a way that the same phrase (``\tsc{sator arepo tenet opera rotas}'', approximately ``Arepo, the farmer, drives carefully the plough'') appears when it is read top-to-bottom, bottom-to-top, left-to-right, and right-to-left.}
\makeatletter
\end{lemma}
\begin{notat}\label{liesin} This allows a certain play for a little abuse of notation, in that we can say that an object $A$ of $\C$ \emph{lies in} a 3-for-2 class $\mathcal K$ if its initial or terminal arrow lies in $\mathcal K$: in this sense, a left normal factorization system is an $\fF$ such that the fiber $KX$ of $X\to RX$ lies in $\E$, for every $X$ in $\C$.
\end{notat}
\end{remark}
Equivalent conditions for $\fF$ to be left normal are given in \cite[Thm. \textbf{4.10}]{RT} and \cite[\textbf{7.3}]{CHK}.
\begin{remark}
There is, obviously, a notion of \emph{right} normal factorization system: it is an $\fF$ such that the cofiber $QX$ of $SX\to X$ lies in $\M$, for every $X$ in $\C$. In the following we call simply \emph{normal}, or \emph{two-sided normal} a factorization system $\fF\in \tsc{fs}(\C)$ which is both left and right normal.
\end{remark}
Rather surprisingly, due to the self-dual setting we are working in, we are able to prove that
\begin{quote}
in a stable $\infty$-category the three notions of simple, semiexact and normal torsion theory collapse to be three equivalent conditions.
\end{quote}
More precisely we have
\begin{proposition}\label{equivcondnorm}
For every object $X$, consider the following diagram in $\C$, where every square is a pullout. 
\[
\xymatrix{
SX\oplus RX[-1]\pp \ar[r]\ar[d]_{m''}& SX \pp \ar[r]\ar[d]^{\sigma_X} & 0\ar[d] \\
KX \pp \ar[r]\ar[d]& X \pp \ar[r]\ar[d]^{\rho_X} & QX\ar[d]^{e''} \\
0 \ar[r]& RX\ar[r] & SX[1]\oplus RX 
}
\]
Then the following conditions are equivalent for a bireflective factorization system $\fF=(\E,\M)$ on $\C$:
\begin{enumerate}
\item $\fF$ is left normal;
\item $\fF$ is right normal;
\item $\fF$ is normal;
\item $RX\simeq QX$; 
\item $SX = KX$; 
\item $SX\to X\to RX$ is a fiber sequence.
\end{enumerate}
\end{proposition}
\begin{proof}
We start by proving that the first three conditions are equivalent. If we assume left normality, then the arrow $\var{QX}{SX[1]\oplus RX}$ lies in $\E$, since it results as a pushout of an arrow in $\E$. So we can consider 
\[
\xymatrix{
QX \ar[rr]^{e'}\ar[d]_{e''} & & RQX\ar[d]^{m'} \\
SX[1]\oplus RX \ar[r]_-{e} & RX = R(SX[1]\oplus RX)\ar[ur]_\sim \ar[r]_-{m} & 0
}
\]
$\fF$-factoring the morphisms involved: $R(SX[1]\oplus RX)=RRX=RX$ since $RS=0$. Thus $RQX\cong RX$, which entails $\var{0}{QX}\in\M$, which entails right normality. A dual proof gives that $(2)\Rightarrow (1)$, thus right normality equals left normality and hence two-sided normality.
Now it is obvious that $(6)$ is equivalent to $(4)$ and $(5)$ together; the non-trivial part of the proof consists of the implications $(1)\Rightarrow (4)$, and dually $(2)\Rightarrow (5)$. 

Once noticed this, start with the diagram
\makeatother
\[
\xymatrix@R=5mm@C=5mm{
SX \ar[rr]^m\ar[dd]&& X \ar[dd]^e\ar[dl] \\
& QX \ar@{.>}[dr]  & \\
0 \ar[rr]_m \ar[ur]&& RX
}
\]
\makeatletter
and consider the canonical arrow $QX\to RX$ obtained by universal property: the arrow $\var{0}{RX}$ lies in $\M$ (this is a general fact); left normality now entails that $\var{0}{QX}\in \M$, so that $\var{QX}{RX}$ lies in $\M$ too by reflectivity.

A similar argument shows that since both $\var{X}{QX}, \var{X}{RX}$ lie in $\E$, $\var{QX}{RX}$ lies in $\E$ too by reflectivity. This entails that $\var{QX}{RX}$ is an equivalence. Conversely, if we start supposing that $QX\cong RX$, then we have (left) normality. This concludes the proof, since in the end we are left with the equality $(4)\iff (5)$.
\end{proof}
As previewed before, the three notions of simplicity, semiexactness and normality collapse in a single notion in the stable setting:
\begin{proposition}\label{simplenormalexact}
A torsion theory $\fF$ is left normal if and only it is semi-left-exact in the  sense of \cite[\textbf{4.3.i}]{CHK}, namely if and only if in the pullout square
\[
\xymatrix{
E \pp \ar[r]\ar[d]_{e'}& X\ar[d]^{\rho_X\in\E} \\
Q \ar[r]_m & RX
}
\]
the arrow $e'$ lies in $\E$. Dually, a factorization system $\fF$ is right normal if and only it is semi-right-exact in the  sense of (the dual of) \cite[\textbf{4.3.i}]{CHK}.
\end{proposition}
\begin{proof}
Consider the diagram
\[
\xymatrix{
KX \pp \ar[r]\ar[d] & E \pp \ar[r]\ar[d]_{e'} & X \ar[d]^e\\
0 \ar[r]& Q\ar[r]_m & RX
}
\]
where the arrow $Q\to RX$ belongs to $\M$. On the one side it is obvious that if $\fF$ is semi-left-exact, then it is normal (just pull back two times $e$ along $\M$-arrows). On the other hand, the converse implication relies on the pullout axiom: if $\fF$ is normal, then $KX$ lies in $\E$; but now since the left square is a pullout, the arrow $\var{E}{Q}$ belongs to $\E$ too, giving semi-left-exactness.
\end{proof}

\begin{remark}
The three notions coincide since ``classically'' we have
\[
\tsc{slex}\to \tsc{simple}\to \tsc{normal},
\]
whereas in our setting the chain of implication proceeds one step further and closes the circle:
\[
\tsc{slex}\to \tsc{simple}\to \tsc{normal}\xto{\star}\tsc{slex}.
\]
This gives a pleasant consequence:
\begin{quote}
In a stable $\infty$-category the $\fF$-factorization of $f\colon A\to B$ with respect to a normal torsion theory is always
\[
A \to RA\times_{RB}B \to B,
\]
or equivalently (see Prop. \refbf{facto})
\[
A \to SB\amalg_{SA} A \to B.
\]
\end{quote}
\end{remark}
We now would like to exploit the theory laid down so far to prove the fundamental result of this work, namely a characterization of $t$-structures as normal torsion theories.
\begin{theorem}\label{thm:rosetta}
Let $\C$ be a stable $\infty$-category. There is a bijective correspondence (in fact, an antitone equivalence of posets) between the class of normal torsion theories $\fF =(\E,\M)$ on $\C$ (in the sense of Definition \refbf{normal}) and the class of $t$-structures on $\C$ (in the sense of Definition \refbf{tistru}).
\end{theorem}
The proof of this result will occupy the rest of the section: to simplify the discussion we will deduce it as a consequence of a number of separate statements.

We first establish the two correspondences between factorization systems and $t$-structures on $\C$. We are obviously led to exploit the fundamental connection (see \S\refbf{fundconn}): given a normal, bireflective factorization system $\fF=(\E,\M)$ on $\C$ we define the two classes $(\C_{\ge 0}(\fF),\C_{<0}(\fF))$ of the $t$-structure $\tee(\fF)$ to be the torsion and torsion-free classes $(0/\E, \M/0)$  associated to $\fF$, in the sense of Definition \refbf{tortorfree}. On the other hand, given a $t$-structure $\tee=(\C_{\ge 0},\C_{<0})$ in the sense of Definition \refbf{tistru}, we have to define classes $\fF(\tee) = (\E(\tee), \M(\tee))$ which form a factorization system. We set:
\[
\E(\tee)=\{f\in \C^{\simplex{1}} \text{ such that $\tau_{<0}(f)$ is an equivalence}\};
\]
\[
\M(\tee)=\{f\in \C^{\simplex{1}} \text{ such that $\tau_{\geq0}(f)$ is an equivalence}\}.
\]

\begin{proposition}
The pair $\tee(\fF)$ is a $t$-structure on $\C$ in the sense of Definition \refbf{tistru}.
\end{proposition}
\begin{proof}
The orthogonality request is immediate by definition of the two classes. As for the closure under positive/negative shifts, $(A\to B)\in\E$ entails that $(A[1]\to B[1])\in\E$ since left classes in factorization systems are closed under (homotopy) colimits in the arrow category (see Prop. \refbf{prop:clos}) and in particular under the homotopy pushout defining the shift $A\mapsto A[1]$ on $\C$. This justifies the chain of implications
\[
X\in \cate C_{\ge 0}(\fF)  \iff \bsmat[[]0\\ \downarrow \\ X\esmat[]]\in\E 
 \Longrightarrow \bsmat[[] 0\\ \downarrow \\ X[1] \esmat[]]\in\E \iff X[1]\in \cate C_{\ge 0}(\fF) .
\]
The case of $\C_{<0}$ is completely dual: since $\M$ admits any limit, $\bsmat[[] X\\ \downarrow \\ 0 \esmat[]]\in\M$ implies that $\bsmat[[] X[-1] \\ \downarrow \\ 0 \esmat[]]\in\M$, so that $\C_{<0}(\fF)[-1]\subset \C_{<0}(\fF)$.

To see that any object $X\in \C$ fits into a fiber sequence 
$
X_{\geq 0}\to X \to X_{<0},
$
with $X_{\geq 0}$ in $\C_{\geq 0}(\fF)$ and $X_{< 0}$ in $\C_{< 0}(\fF)$, 
it suffices to $\fF$-factor the terminal morphism of $X$ obtaining a diagram like
\makeatother \[
\xymatrix{
X \ar[r]^e & RX \ar[r]^m & 0
}
\]\makeatletter
and then to take the fiber of $e$,
\makeatother \[
\xymatrix{
KX\ar[r]\ar[d]\pp & X\ar[d]  \\
0 \ar[r] & RX  
}
\]\makeatletter
Set $X_{\geq 0}=KX$ and $X_{<0}=RX$. Then $X_{<0}\in \C_{<0}(\fF)$ by construction and $X_{\geq 0}\in \C_{\geq 0}(\fF)$ by normality.
\end{proof}
In order to prove that the pair of markings $\fF(\tee)$ is a factorization system on the stable $\infty$-category $\C$, we use the data of the $t$-structure to produce a functorial factorization of morphisms.
To do this, recall that by Definition \refbf{tistru}\textbf{.(iii)} every object $X\in\C$ fits into a fiber sequence (a ``distinguished triangle'') $X_{\ge 0} \to X\to X_{<0}\to X_{\ge 0}[1]$. So, given $f\colon X\to Y$ we can build the diagram\footnote{We thank Eric Wofsey for having suggested us to consider this diagram \cite{Wofsey}.}
\makeatother
\[\tag{$\star\star$}
\xymatrix{
X_{\ge 0}\ar[r]\ar[d]_{\tau_{\ge 0}(f)} & X\ar@/^1pc/@[gray][dd]|\hole \ar[r]\ar@{.>}[d]_{e_f}& X_{<0}\ar[r]\ar@{=}[d] & X_{\ge 0}[1]\ar[d]^{\tau_{\ge 0}(f)[1]} \\
Y_{\ge 0} \ar@{.>}[r]\ar@{=}[d]& C \pp \ar[r]\ar[d]_{m_f}& X_{<0}\ar[r]\ar[d]^{\tau_{<0}(f)} & Y_{\ge 0}[1] \ar@{=}[d]\\
Y_{\ge 0} \ar[r]& Y\ar[r] & Y_{<0}\ar[r] & Y_{\ge 0}[1] 
}
\]
\makeatletter
where the decorated square is a pullout (so $C\cong X_{<0}\times_{Y_{<0}}Y$, a characterization which, alone, should be reminiscent of simplicity for the would-be factorization of $f$: cleaning up the above diagram a bit we can recognize precisely the same diagram of Definition \refbf{simple}, up to the identifications $\tau_{<0}=R$ and $\tau_{\ge 0}=S$), and hence the dotted arrows are determined by the obvious universal property. Note that all the three rows in the above diagram are fiber sequences. Mapping $f$ to the pair $(e_f,m_f)$ is a factorization functor $F\colon \C^{\simplex{1}}\to \C$ (a tedious but easy check) in the sense of \cite{Korostenski199357} (see also our \S \refbf{korostenski}). Next, we invoke a rather easy but subtle result contained in \cite{Korostenski199357}, which in a nutshell says that a factorization system on a category $\C$ is determined by a functorial factorization $F$ 
such that $m_{e_f}$, $e_{m_f}$ are invertible. Functors satisfying this property are called \emph{Eilenberg-Moore factorization functors} in \cite{Korostenski199357}.\footnote{These are not the weakest assumptions to ensure that $\fF(F)=(\E_F, \M_F)\in \tsc{fs}(\C)$: see the final remark in \cite{Korostenski199357} and \cite[\textbf{1.3}]{janelidze1999functorial}.} Namely, if one defines
\[
\E_F = \{h\in \C^{\simplex{1}}\mid m_h \text{ is invertible}\}
\]
and 
\[
\M_F = \{h\in \C^{\simplex{1}}\mid e_h \text{ is invertible}\},
\]
then $(\E_F,\M_F)$ is a factorization system as soon as $e_f\in \E_F$ and $m_f\in \M_F$ for any morphism $f$ in $\C$.

\begin{remark}
Before we go on with the proof notice that by the very definition of the factorization functor $F$ associated with a $t$-structure above, we have that $\M_F$ coincides with the class of arrows $f$ such that the naturality square of $f$ with respect to the ``truncation'' functor $\tau_{<0}$ of the $t$-structure is cartesian: we denote this marking of $\C$ as $\tsc{Cart}(\tau_{<0})$ adopting the same notation as \cite[\S \textbf{4}]{RT}.
\end{remark}
The following lemma is the $t$-structure counterpart of Proposition \refbf{facto}.
\begin{lemma}
\label{another-pullout}The homotopy commutative sub-diagram 
\[
\xymatrix{
X_{\ge 0}\ar[r]\ar[d]_{\tau_{\ge 0}(f)} & X\ar[d]^{e_f} \\
Y_{\ge 0} \ar[r]& C 
}
\]
in the diagram ($\star\star$) is a pullout.
\end{lemma}
\begin{proof}
Consider the diagram
\[
\xymatrix{
X_{\geq 0}\ar[r]\ar[d]_{\tau_{\geq 0}(f)}&X\ar[d]^{e_f}\\
Y_{\ge 0} \ar[r]\ar[d]& C \pp \ar[d]\ar[r]^{m_f}& Y\ar[d]\\
0 \ar[r]& X_{<0}\ar[r]_{\tau_{<0}(f)} & Y_{<0}
}
\]
where all the squares are homotopy commutative and apply twice the 2-for-3 law for pullouts.
\end{proof}
\begin{lemma}
Let $F:f\mapsto(e_f,m_f)$ be the factorization functor associated with a $t$-structure by the diagram ($\star\star$). Then $\tau_{<0}(e_f)$ and $\tau_{\geq 0}(m_f)$ are equivalences. 
\end{lemma}
\begin{proof}
Since $\tau_{<0}\tau_{\ge 0}=0$, by applying $\tau_{<0}$ to the pullout diagram in $\C$ given by lemma \refbf{another-pullout}, we get the pushout diagram
\[
\xymatrix{
0\po\ar[r]\ar[d] & X_{<0}\ar[d]^{\tau_{<0}(e_f)} \\
0 \ar[r]& C_{<0} 
}
\]
in $\C_{<0}$ which tells us that $\tau_{<0}(e_f)$ is an equivalence. The proof that $\tau_{\ge 0}(m_f)$ is a equivalence is perfectly dual and is obtained by applying $\tau_{\ge 0}$ to the marked pullout diagram in ($\star\star$).
 \end{proof}
It is now rather obvious that showing that
\[
\E_F = \tau_{<0}^{-1}(\iso); \qquad \M_F = \tau_{\geq 0}^{-1}(\iso)
\]
will imply that $F$ is an Eilenberg-Moore factorization functor. 
Once proved this, it is obvious that the preimage of a 3-for-2 class along a functor is again a 3-for-class in $\C$, and this entails that both classes in $\fF(\tee)$ are 3-for-2. We are now ready to prove
\begin{proposition}
The pair of markings $\fF(\tee)$ is a factorization system on the quasicategory $\C$, in the sense of Definition \refbf{def:effe-esse}.
\end{proposition}
\begin{proof}
By the very definition of the factorization procedure, and invoking the pullout axiom, we can deduce that 
the arrow $f$ lies in $\E_F$ if and only if it is inverted by $\tau_{<0}$; this entails that $\E_F = \tau_{<0}^{-1}(\iso)$. So it remains to show that $\M_F = \tau_{\ge 0}^{-1}(\iso)$. We have already remarked that $\M_F=\tsc{Cart}(\tau_{<0})$, so we are reduced to showing that $ \tau_{\ge 0}^{-1}(\iso)=\tsc{Cart}(\tau_{<0})$. But again, this is easy because on the one side, if $f\in \tsc{Cart}(\tau_{<0})$ then the square
\[
\xymatrix{
\ar[r]\ar[d]_{\tau_{\ge 0}(f)} \pp & \ar[d]^{\tau_{\ge 0}\tau_{<0}(f)}\\
\ar[r] &
}
\]
is a pullout since $\tau_{\geq_0}$ preserves 
pullouts, and yet $\tau_{\ge 0}\tau_<0(f)$ is the identity of the zero object. So $\tau_{\ge 0}(f)$ must be an equivalence. On the other hand, the stable $\infty$-categorical analogue of the triangulated 5-lemma (see \cite[Prop. \textbf{1.1.20}]{Nee}), applied to the diagram ($\star\star$) shows that if $\tau_{\ge 0}(f)$ is an equivalence then $e_f$ is an equivalence and so $C\cong X$, i.e., $f\in \tsc{Cart}(\tau_{<0})$.
\end{proof}
\begin{remark}
As a side remark, we notice that a completely dual proof would have arisen using $C=Y_{\geq 0}\amalg_{X_{\ge 0}}X$ (see Lemma \refbf{another-pullout}) and then showing first that $\fF(\tee)$ is the factorization system $(\tsc{Cocart}(\tau_{\ge 0}), \tau_{\ge 0}^{-1}(\iso))$ and then $\tsc{Cocart}(\tau_{\ge 0})=\tau_{<0}^{-1}(\iso)$. 
\end{remark}
To check that $\fF(\tee)$ is normal, it only remains to verify that any of the equivalent conditions for normality given in Proposition \refbf{equivcondnorm} holds, which is immediate.
This concludes the proof that there is a correspondence between normal torsion theories and $t$-structures: it remains to show that this correspondence is bijective, i.e., that the following proposition holds.
\begin{proposition}
In the notations above, we have $\fF(\tee(\fF))=\fF$ and $\tee(\fF(\tee))=\tee$.
\end{proposition}
\begin{proof}
On the one side, consider the factorization system $\fF(\tee(\fF)) =(\tau_{<0}^{-1}(\iso)$, $\tau_{\ge 0}^{-1}(\iso))$, where the functor $\tau_{<0}$ is defined starting from the $\fF$-factorization of each $X\to 0$, as in the fundamental connection of \S\refbf{fundconn}: $X\xto{e}X_{<0}\xto{m}0$. Recall (Remark \refbf{funtoriali}) that the action of $\tau_{<0}\colon \C\to \M/0$ on arrows is obtained from a choice of solutions to lifting problems
\[
\xymatrix{
A \ar[r]^{e'f}\ar[d]_e & \tau_{<0}B\ar[d]^{m'} \\
\tau_{<0}A \ar[ur]_{\tau_{<0}(f)}\ar[r]_{m}& 0.
}
\]
It is now evident that $\tau_{<0}^{-1}(\iso) = \E$. Indeed:
\begin{itemize}
\item If $f\in \tau_{<0}^{-1}(\iso)$, then in the above square $e'f=\tau_{<0}(f)\, e$, which is in $\E$ since $\E$ contains equivalences and is closed for composition. But $e'$ lies in $\E$, so that $f\in\E$ by the 3-for-2 property of $\E$;
\item If $f\in \E$, then $e'f$ is in $\E$ and so in the same square we read two lifting problems with unique solutions, which implies that $\tau_{<0}(f)$ is invertible.
\end{itemize}
On the other side, we have to compare the $t$-structures $\tee = (\C_{\ge 0}, \C_{<0})$ and $\tee(\fF(\tee))$. We have $X\in \C_{\geq 0}(\fF(\tee))$ if and only if $\var{0}{X}\in \E(\tee)$. Since $\E(\tee) = \tau_{<0}^{-1}(\iso)$, we see that $X\in \C_{\ge 0}(\fF(\tee))$ if and only if  $X_{<0}\cong 0$. But it is a direct consequence of Lemma \refbf{orthoreflex} that $X_{<0}\cong 0$  if and only if $X\in \C_{\ge 0}$. Dually, one proves that $\C_{<0}(\fF(\tee))=\C_{<0}$.\end{proof}
\section{Selected exercises.}
\epigraph{\vbox{\lettrine{Q}{uinto} exercicio es meditaci\'on del Infierno.}\tiny{\,\,\,\,\hskip 5 pt.}}{\'I.~ L.~ de Loyola, \emph{Exercicios espirituales}}

The factorization systems point of view can be usefully employed to prove the stable $\infty$-categorical version of a few classical results on $t$-structures in triangulated categories, which appear to be missing a detailed discussion in \cite{LurieHA}. Here we propose these results in the form of exercises on which the reader can test the familiarity they have gained with the constructions presented in the main body of this note. A detailed discussion appears in \cite{heart}.
\begin{exercise}[The heart of a $t$-structure is abelian] The \emph{heart} of a $t$-structure $\tee=(\C_{\ge 0}, \C_{<0})$ on a stable $\infty$-category $\C$ is the full subcategory of $\C$ given by the intersection $\C^\heart=\C_{[0,1)}=\C_{\ge 0}\cap \C_{<1}$. Prove that $\C^\heart $ is an abelian $\infty$-category and so in particular its homotopy category is an abelian category (this was first proved in \cite[Thm. \textbf{1.3.6}]{BBDPervers} for triangulated categories, and is quoted without proof in \cite[Remark \textbf{1.2.1.12}]{LurieHA}).
\end{exercise}
\noindent{\it Hint:} Define the kernel of a morphism $f\colon X\to Y$ in $\C^\heart$ as $\ker(f)=(\mathrm{fib}(f))_{\geq 0}$ and the cokernel of $f$ as $\mathrm{coker}(f)=(\mathrm{cofib}(f))_{<1}$.
\begin{remark}\label{evocative}
There is a rather evocative pictorial representation of the heart of a $t$-structure, manifestly inspired by \cite{Brid}:
 if we depict $\C_{<0}$ and $\C_{\geq 0}$ as contiguous half-planes, like in the following picture,
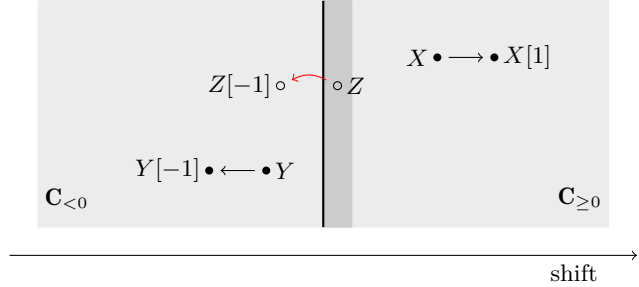
\begin{center}
\begin{tikzpicture}[scale=.75]
\filldraw[gray!15] (5,-2) -- (-5,-2) -- (-5,2) -- (5, 2) -- cycle;
\filldraw[gray!40] (.5, -2) -- (0,-2) -- (0,2) -- (.5,2) -- cycle;
\draw[thick] (0,-2) -- (0,2);
\fill (2,1) circle (2pt) node[left] (X) {$X$};
\fill (-1,-1) circle (2pt) node[right] (Y) {$Y$};
\node at (4.5,-1.5) {$\C_{\geq0}$};
\node at (-4.5,-1.5) {$\C_{<0}$};
\draw[->] (2.2,1) -- (2.8,1);
\draw[->] (-1.2,-1) -- (-1.8,-1);
\fill[xshift=1cm] (2,1) circle (2pt) node[right] (X) {$X[1]$};
\fill[xshift=-1cm] (-1,-1) circle (2pt) node[left] (Y) {$Y[-1]$};
\draw (.25,0.5) circle (2pt) node[right] {$Z$};
\draw (-.75,0.5) circle (2pt) node[left] {$Z[-1]$};
\draw[->, red] (.05,0.6) to[bend right] (-.55,0.6);
\draw[->, xshift=-.5cm, yshift=-.5cm] (-5,-2) -- (6,-2) node[below, pos=.9] {\text{shift}};
\end{tikzpicture}
\end{center}
then the action of the shift is \emph{precisely} an horizontal shift, and the closure properties of the two classes $\C_{\geq0},\C_{<0}$ under positive and negative shifts are a direct consequence of their shape. With these notations, an object $Z$ is in the heart of $\tee$ if it lies in the ``shadowed region'', i.e. if it lies in $\C_{\geq0}$, but $Z[-1]$ lies in $\C_{<0}$.
\end{remark}

\begin{exercise}[Postnikov towers]\label{postnikov} Let $(\C,\tee)$ be a stable $\infty$-category endowed with a $t$-struc\-tu\-re. A morphism $f\colon X\to Y$ in $\C$ is said to be \emph{bounded} with respect to $\tee$ if there exist integers $a<b$ such that $\mathrm{fib}(f)\in\C_{[a,b)}=\C_{\geq a}\cap \C_{<b}$. Prove that, if $f$ is bounded then there exist a factorization of $f$
\[
X\simeq Z_0\xrightarrow{f_0}Z_{1}\xrightarrow{f_{1}}Z_{2}\xrightarrow{f_{2}}\cdots\xrightarrow{f_{b-a-1}}Z_{b-a}\simeq Y
\] 
with $\mathrm{fib}(f_k)\in \C^\heart[a+k]$. This factorization is called the Postnikov tower of $f$ and, for $f$ an initial morphism this is \cite[Prop. \textbf{1.3.13}]{BBDPervers} or \cite[Lemma \textbf{3.2}]{Brid}.
\end{exercise}
\noindent{\it Hint:} Use a shift to reduce to the case $a=0$ and then use induction on $b-a$.

\begin{exercise}[$t$-structures from Postnikov towers] Prove that the following converse of the result stated in Exercise \refbf{postnikov} holds. Let $\mathbf{H}$ be an abelian $\infty$-subcategory of the stable $\infty$-category $\C$. 
If for any morphism $f$ in $\C$ there exist integers $a<b$ and a functorial factorization  of $f$
\[
X\simeq Z_0\xrightarrow{f_0}Z_{1}\xrightarrow{f_{1}}Z_{2}\xrightarrow{f_{2}}\cdots\xrightarrow{f_{b-a-1}}Z_{b-a}\simeq Y
\] 
with $\mathrm{fib}(f_k)\in \mathbf{H}[a+k]$, then there exists a $t$-structure $\tee$ on $\C$ with $\C^\heart=\mathbf{H}$ and such that every morphism in $\C$ is $\tee$-bounded. This is \cite[Lemma \textbf{3.2}]{Brid}
\end{exercise}
\noindent{\it Hint:} Use the factorization $f\mapsto(f_0,f_1,\dots,f_{b-a-1})$ of an initial morphisms $0\to Y$ to decide which objects $Y$ should belong to the subcategory $\C_{\geq 0}$, and the factorization of terminal morphisms $X\to 0$ to decide which objects $X$ should belong to the subcategory $\C_{<0}$.
\medskip

Finally, we propose an exercise related to the notion of slicing in a stable $\infty$-category. 
\begin{exercise}
Recall from Remark \refbf{slicing} that a slicing on a stable $\infty$-category $\C$ is a collection $(\C_{\geq t},\C_{<t})_{t\in \mathbb{R}}$ of $t$-structures with:
\begin{itemize}
\item $\C_{<t_1}\subseteq \C_{<t_2}$ if $t_1\leq t_2$;
\item $\C_{<t+1}=\C_{<t}[1]$, for any $t\in \mathbb{R}.$
\end{itemize}
For any $\epsilon\in \mathbb{R}$ with $0<\epsilon<1$, let $\C_{[0,\epsilon)}=\C_{\geq0}\cap \C_{<\epsilon}$. Does $\C_{[0,\epsilon)}$ have kernels and cokernels? Is $\C_{[0,\epsilon)}$ an abelian $\infty$-category?
\end{exercise}
\noindent{\it Hint:} See \cite[Section 4]{Brid}.
          
\begin{acknowledgements} 
The authors would like to thank 
Paolo {Brasolin}, 
Sergio Buschi,
Olivia {Caramello}, 
Andr\'e Joyal, Urs Schreiber, 
all the $n$Labbers and \texttt{MathOverflow} members, and the referee,  
for stimulating discussions, crucial suggestions, insights, comments, encouragement and many other things that would be too long to enumerate here, and in particular Eric {Wofsey} for having provided one of the key-steps in the proof of our main theorem. F. L. would like to thank Emily {Riehl} and all the ``Kan extension seminar'' crew for having allowed him to learn thoroughly the theory of factorization systems.
\end{acknowledgements}
%
%
%
%
\providecommand{\bysame}{\leavevmode\hbox to3em{\hrulefill}\thinspace}
\providecommand{\MR}{\relax\ifhmode\unskip\space\fi MR }
\providecommand{\MRhref}[2]{%
  \href{http://www.ams.org/mathscinet-getitem?mr=#1}{#2}
}
\providecommand{\href}[2]{#2}

\hrulefill 
\end{document}